\documentclass[%
  abstract=true,%
  smallheadings,%
]{scrartcl}

\usepackage[T1]{fontenc}
\usepackage[utf8]{inputenc}
\usepackage{lmodern}
\usepackage{amsmath,amsfonts,amsthm,amssymb}
\usepackage{mathtools}
\usepackage{todonotes}
\usepackage[sort&compress,numbers]{natbib}
\usepackage{overpic}
\usepackage[outline]{contour}
\contourlength{1pt}
\usepackage{microtype}
\usepackage{booktabs}

\setkomafont{sectioning}{\normalfont\bfseries}
\setkomafont{descriptionlabel}{\itshape}

\usepackage[hidelinks]{hyperref}
\hypersetup{
  pdfauthor={Zijia Li, Josef Schicho, Hans-Peter Schröcker},%
  pdftitle={Spatial straight line linkages by factorization of motion polynomials},%
}

\theoremstyle{definition}
\newtheorem{lemma}{Lemma}

\theoremstyle{theorem}
\newtheorem{theorem}{Theorem}
\theoremstyle{remark}
\newtheorem{remark}{Remark}

\newcommand{\SE}[1][3]{\mathrm{SE}(#1)}

\newcommand{\RSet}{\mathbb{R}}
\newcommand{\DSet}{\mathbb{D}}
\newcommand{\HSet}{\mathbb{H}}
\newcommand{\DHSet}{\DSet\HSet}
\newcommand{\eps}{\varepsilon}

\newcommand{\cj}[1]{\overline{#1}}

\newcommand{\qi}{\mathbf{i}}
\newcommand{\qj}{\mathbf{j}}
\newcommand{\qk}{\mathbf{k}}

\newcommand{\peq}{\equiv}

\title{Spatial straight line linkages by factorization of motion polynomials}
\author{Zijia Li\footnote{%
    Johann Radon Institute for Computational and Applied Mathematics,
    Austrian Academy of Science,
    \href{mailto:zijia.li@oeaw.ac.at}{zijia.li@oeaw.ac.at},
    \href{mailto:josef.schicho@risc.jku.at}{josef.schicho@risc.jku.at}} %
  \and Josef Schicho\footnotemark[1] %
  \and Hans-Peter Schröcker\footnote{%
    Unit Geometry and CAD, University of Innsbruck,
  \href{hans-peter.schroecker@uikb.ac.at}{hans-peter.schroecker@uibk.ac.at}}}

\begin{document}

\maketitle

\begin{abstract}
  We use the recently introduced factorization of motion polynomials
  for constructing overconstrained spatial linkages with a straight
  line trajectory. Unlike previous examples, the end-effector motion
  is not translational and the link graph is a cycle. In particular,
  we obtain a number of linkages with four revolute and two prismatic
  joints and a remarkable linkage with seven revolute joints one of
  whose joints performs a Darboux motion.
\end{abstract}

\par\noindent\emph{Keywords:} Single loop linkage, %
                              revolute joint, %
                              prismatic joint, %
                              straight line trajectory, %
                              Darboux linkage
\par\noindent\emph{MSC 2010:}
70B10 

\section{Introduction}
\label{sec:introduction}

Spatial mechanisms with exact straight line trajectories are rare. The
first non-trivial example is due to \cite{sarrus53}. It has the
property that all trajectories are straight lines and is nowadays
called Sarrus' 6R linkage. Multi-looped linkages, composed of
spherical and planar parts, with one straight line trajectory were
presented by Pavlin and Wohlhart in \cite{pavlin92}. Other mechanisms
with non-trivial straight line trajectories include the ``Wren
platform'' and some of its variants \cite{kiper11,wohlhart96} or the
generators for the vertical Darboux motion of Lee and Hervé
\cite{lee12}.

In this article we construct new single-looped linkages with a
straight line trajectory. In contrast to Sarrus' linkage, the
end-effector motion is not purely translational. In contrast to the
examples given by Pavlin and Wohlhart, the linkage is single-looped
and in general not composed of planar or spherical parts. In a special
case, we show that the Darboux motion can be uniquely decomposed in a
rotation and a circular translation and use this for the construction
of Darboux linkages which do not contain prismatic or cylindrical
joints and, in contrast to \cite{lee12}, perform the general Darboux
motion. To define the scope of this paper more precisely: We
systematically construct closed-loop straight line linkages with only
revolute or prismatic joints whose coupler motion is neither planar,
nor spherical, nor translational and has degree three in the dual
quaternion model of rigid body displacements.

We do not claim that spatial straight line linkages are of particular
relevance to engineering sciences. But it should be evident after
reading this paper that we gained new insight to some well-known
planar and spatial motions. The presented ideas may be extended to
other, more useful, synthesis tasks. Our basic tool is factorization
of motion polynomials, as introduced in \cite{hegedus13}. While that
paper presents a general theory and algorithmic treatment for the
generic case, a good deal of this paper deals with non-generic cases
and thus furthers our understanding of motion polynomial
factorization.

\section{Preliminaries}
\label{sec:preliminaries}

We continue with a brief introduction to the dual quaternion model of
rigid body displacements. In particular, we derive the straight line
constraint in that model and introduce the notion of ``motion
polynomials''.

\subsection{The straight line constraint}
\label{sec:straight-line-constraint}

We begin be deriving the constraint equation for all direct isometries
of Euclidean three-space that map one point $p$ onto a straight line
$L$. We do this in terms of dual quaternions, making use of the
well-known isomorphism between the group $\SE$ of direct isometries
and the factor group of unit dual quaternions modulo $\pm 1$. A dual
quaternion is an expression of the shape
\begin{equation*}
  h = h_0 + h_1\qi + h_2\qj + h_3\qk + \eps(h_4 + h_5\qi + h_6\qj + h_7\qk).
\end{equation*}
Multiplication of dual quaternions is defined by the rules
\begin{equation*}
  \qi^2 = \qj^2 = \qk^2 = \qi\qj\qk = -1,\quad
  \qi\eps = \eps\qi,\quad
  \qj\eps = \eps\qj,\quad
  \qk\eps = \eps\qk.
\end{equation*}
We denote the set of dual quaternions by $\DHSet$. The dual quaternion
$h$ may be written as $h = p + \eps q$ with ordinary quaternions $p$,
$q \in \HSet$, its the \emph{primal} and \emph{dual part.}

After projectivizing $\DHSet$, we obtain Study's kinematic mapping
$\SE \to P^7$, see for example \cite{husty12}. The unit dual
quaternion $x + \eps y$ acts on $p = (p_0,p_1,p_2) \in \RSet^3$
according to
\begin{equation}
  \label{eq:1}
  1 + \eps (p_1'\qi + p_2'\qj + p_3'\qk) = (x - \eps y)(1 + \eps
  (p_0\qi + p_1\qj + p_2\qk))(\cj{x + \eps y}).
\end{equation}
The dual quaternion $x + \eps y$ is projectively equal to a unit norm
dual quaternion, if the Study condition $x\cj{y} + y\cj{x} = 0$ is
fulfilled. The action of $x + \eps y$ on $p$ is still defined as in
\eqref{eq:1} but the right-hand side has to be divided by $x\cj{x}$.
It is hence a rational expression in the components of $x$ and $y$.

Straight line constraints in the dual quaternions setting are the
topic of \cite[Section~5.1]{selig11}. We re-derive a dual quaternion
condition for a particular case. Choosing appropriate Cartesian
coordinates in the moving frame, we may assume $p = (0,0,0)$.
Similarly, it is no loss of generality to assume that $\{(t,0,0) \mid
t \in \RSet \}$ is the set of points on $L$. Writing $x = x_0 + \qi
x_1 + \qj x_2 + \qk x_3$ and $y = y_0 + \qi y_1 + \qj y_2 + \qk y_3$,
the second and third coordinate of $p'$ vanish if and only if
\begin{equation}
  \label{eq:2}
  x_0 y_2 - x_1 y_3 - x_2 y_0 + x_3 y_1 = 0,
  \quad
  x_0 y_3 + x_1 y_2 - x_2 y_1 - x_3 y_0 = 0.
\end{equation}
This system has to be augmented with the Study condition
\begin{equation}
  \label{eq:3}
  x_0y_0 + x_1y_1 + x_2y_2 + x_3y_3 = 0.
\end{equation}
It is straightforward to check that the system of
equations~\eqref{eq:2} and \eqref{eq:3} has the solution
\begin{equation}
  \label{eq:4}
  x \peq \qi y
  \quad
  \text{or, equivalently,}
  \quad
  y \peq -\qi x
\end{equation}
where ``$\peq$'' denotes equality in projective sense, that is, up to
multiplication with constant scalars.

\subsection{Motion polynomials}
\label{sec:motion-polynomials}

Denote the set of all polynomials in the indeterminate $t$ by
$\DHSet[t]$ and, similarly, by $\RSet[t]$ the set of polynomials in
$t$ with real coefficients. A parameterized rational motion is given
by a polynomial $C = X + \eps Y \in \DHSet[t]$ with the property
$X\cj{Y} + Y\cj{X} = 0$ or, equivalently, $C\cj{C} \in \RSet[t]$ (the
conjugate polynomial is obtained by conjugating the
coefficients). These polynomials have been called \emph{motion
  polynomials} in \cite{hegedus13}. Their coefficients are dual
quaternions and do not commute. Therefore, additional conventions for
notation, multiplication and evaluation are necessary:
\begin{itemize}
\item We always write coefficients to the left of the indeterminate
  $t$. This convention is sometimes emphasized by speaking of
  ``left-polynomials'' but we just use the term
  ``polynomial''.
\item Multiplication of polynomials uses the additional rule that the
  indeterminate $t$ commutes with all coefficients.
\item The value of the polynomial $C = \sum_{i=0}^n c_it^i$ at $h \in
  \DHSet$ is defined as $C = \sum_{i=0}^n c_ih^i$, that is, it is
  obtained by substituting $t$ by $h$ in the \emph{expanded form}.
\end{itemize}
Here is a short example to clarify these conventions. Consider the
polynomial $C = (t-h)(t-k)$ with $h$, $k \in \DHSet$. Its expanded
form reads $C = t^2 - (h + k)t + hk$ (we used commutativity of $t$ and
$k$). The dual quaternion $k$ is a zero of $C$ but $h$ is, in general,
not:
\begin{equation*}
  C(k) = k^2 - (h + k)k + hk = 0,\quad
  C(h) = h^2 - (h + k)h + hk = hk - kh.
\end{equation*}
Substituting $t$ by $h$ in the factorized form gives a different
value. This is clear since factorized form and expanded form are only
equivalent under commutativity assumptions.

Above examples suggest a relation between right factors and zeros of
motion polynomials that, in fact, holds true in a more general
setting. The following lemma has been stated in proved in
\cite[Lemma~2]{hegedus13}.
\begin{lemma}
  \label{lem:1}
  Let $P \in \DHSet[t]$ and $h \in \DHSet[t]$. Then $t-h$ is a right
  factor of $P$ (there exists $Q \in \DHSet$ such that $P = Q(t-h)$)
  if and only if $P(h) = 0$.
\end{lemma}

In order to apply motion factorization for the construction of
straight line linkages, we need to find a polynomial
$C = X + \eps Y \in \DHSet[t]$ that satisfies \eqref{eq:4} identically
in $t$. This already implies that $C$ is a motion polynomial. Our
construction of straight line linkages is largely based on the
factorization theorem for motion polynomials
\cite[Theorem~1]{hegedus13}. This theorem states, that a monic motion
polynomial of degree $n$ generically admits $n!$ factorizations of the
shape
\begin{equation}
  \label{eq:5}
  C = (t-h_1) \cdots (t-h_n)
\end{equation}
with $h_i \in \DHSet$ representing rotations or translations. 

The algorithm for computing factorizations in generic cases is
explained in \cite{hegedus13} and, in more algorithmic form, in
\cite{hegedus13b}. A basic understanding of this algorithm is
necessary for reading this paper. Therefore, we provide an informal
but detailed description. A more formal algorithmic description in
pseudo-code is given in \cite{hegedus13b}, actual implementations can
be found in the supplementary material of~\cite{hegedus13}.

The norm polynomial $C\cj{C}$ is real and factors into the product
$C\cj{C} = M_1\cdots M_n$ of $n$ quadratic factors. Each factor $M_i$
is either irreducible over $\RSet$ or the square of a linear factor.
In order to compute a factorization of the shape \eqref{eq:5}, we pick
one of the quadratic factors, say $M_i$, and right-divide $C$ by
$M_i$. That is, we compute $Q,R \in \DHSet[t]$ such that $\deg R \le
1$ and $C = QM_i + R$. In general, $R$ has a unique zero -- the
rotation or translation polynomial $h_n$. Once the rightmost factor
$h_n$ has been computed, we compute $P_1$ such that $P = P_1(t-h_n)$ and
repeat above steps with $P_1$ instead of $P$. Note that
\begin{equation*}
  P_1\cj{P_1} = \prod_{j \neq i} M_j
\end{equation*}
such that all factors of the original norm polynomial $C\cj{C}$ will
be used during this process. In this sense, we can say that a factor
$t-h_i$ or the rotation/translation quaternion $h_i$ itself
``corresponds'' to a factor $M_j$. Different factorizations come from
permutations of the factors of $C\cj{C}$.

In exceptional cases, the leading coefficient of the linear remainder
polynomial $R$ fails to be invertible. Then, above algorithm will not
produce a valid factorization. This does, however, not mean that no
factorization exists. In fact, in this paper we will encounter
situations with no or infinitely many factorizations of the shape
\eqref{eq:5}.

The kinematic interpretation of motion polynomial factorization is
that the motion polynomial parameterizes the rational end-effector
motion of, in general, $n!$ open chains consisting of $n$ revolute or
prismatic joints. Linkages are obtained by suitably combining a
sufficient number of these open chains. In case of $\deg C \le 3$, two
suitably chosen open chains are in general sufficient and will result in
an overconstrained, single-looped linkage. Constructions of this type
are the topic of this paper's main section.

\section{Mechanism synthesis}
\label{sec:mechanism-synthesis}

The most general polynomial solution of \eqref{eq:4} is given by
\begin{equation*}
  C = X + \eps Y\quad\text{with}\quad
  X = \xi P;\quad
  Y = -\eta \qi P;\quad
  P \in \HSet[t];\quad
  \xi,\ \eta \in \RSet[t]
\end{equation*}
Let us verify that the trajectory of $p = (0,0,0)$ is really a
straight line. According to \eqref{eq:1}, the image $p'$ of $p$ can be
read off from
\begin{equation}
  \label{eq:6}
  \begin{aligned}
    1 + p'  &\peq (X - \eps Y)(\cj{X + \eps Y})
             = (\xi P + \eps \eta \qi P)(\cj{\xi P - \eps \eta \qi P}) \\
            &= (\xi P + \eps \eta \qi P)(\xi \cj{P} + \eps \eta \cj{P} \qi)
             = P\cj{P}(\xi^2 + 2 \eps\xi\eta\qi).
  \end{aligned}
\end{equation}
Indeed, the right-hand side of \eqref{eq:6} leads to a point on the
line $L$. More precisely, a parameterized equation of the trajectory is
\begin{equation*}
  p'(t) = \frac{2\eta}{\xi}\qi.
\end{equation*}
From this, we conclude that $\eta = 0$ or constant $\xi$ and $\eta$
yield a constant point $p'$. The resulting motion is spherical and
shall be excluded from further consideration. That is, we can assume
$\eta \neq 0$ and $\xi$, $\eta$ are not both constant. This implies
$\deg P < \deg C$. In order to narrow the focus of this paper, we also
wish to avoid $\deg P = 0$ or, more generally, $P \in \RSet[t]$. This
leads to a translation in constant direction\,---\,a motion which is
planar in multiple ways.\footnote{Note however, that the factorization
  of a translation in constant direction not necessarily lead to
  planar linkages. An example of this are Sarrus linkages with
  rational coupler motion.} By a change of coordinates we can achieve
that $C$ is monic whence $\deg \eta < \deg \xi$. Finally, we may
transfer constant real factors between $P$ and $\xi$, so that we can
assume that both, $P$ and $\xi$ are monic. Summarizing these
constraints, we have:
\begin{equation*}
  0 \le \deg \eta < \deg \xi,\quad
  1 \le \deg P < \deg C \le 3,\quad
  P \notin \RSet[t],\quad
  \xi, P \text{ are monic.}
\end{equation*}
Hence, we only have to discuss two major cases, $\deg P = 1$ and $\deg
P = 2$. The former has three sub-cases ($\deg \xi = 1$ and $\deg \eta
= 0$, $\deg \xi = 2$ and $\deg \eta = 0$, $\deg \xi = 2$ and $\deg
\eta = 1$), the latter only one ($\deg \xi = 1$, $\deg \eta = 0$).

\subsection{The case of degree two}
\label{sec:case-degree-two}

We consider the case $\deg P = 2$, $\deg \xi = 1$, and $\deg \eta = 0$
first. The norm polynomial admits the factorization $C\cj{C} =
M_1M_2M_3$ where $ \xi^2 = M_1$ and $P\cj{P} = M_2M_3$. This is
already a special case as one factor, $M_1$, is not strictly positive.
The following theorem gives a relation between the factors of a motion
polynomial and the factors of its norm polynomial for this case.

\begin{theorem}
  \label{th:1}
  The norm polynomial of a motion polynomial factors as $C\cj{C} =
  \prod_{i=1}^n M_i$ with non-negative factors $M_1,\ldots,M_n$ which
  are either irreducible over $\RSet$ or the squares of linear
  polynomials in $\RSet[t]$. If $M$ is such a square, the
  corresponding factor $t-h$ in every factorization of $C$ describes a
  translation.
\end{theorem}

The first part of this proposition is already due to \cite{hegedus13}.
The statement on the translation can also be found there but it is
only motivated, not proved.

\begin{proof}[Proof of \autoref{th:1}]
  If $t-h$ is a factor corresponding to $M$, the dual quaternion $h$
  is necessarily a common zero of $C$ and $M$
  (\cite[Lemma~3]{hegedus13}). In particular, if $M = (t-r)^2$ with $r
  \in \RSet$, we can evaluate the condition
  \begin{equation*}
    h^2 - 2hr + r^2 = 0.
  \end{equation*}
  By \cite[Theorem~2.3]{huang02}, this equation can only be satisfied
  by dual quaternions of primal part $r \in \RSet$. Hence, $h$ is
  necessarily a translation quaternion.
\end{proof}

By \autoref{th:1}, every factorization of $C$ contains at least one
prismatic joint, corresponding to $M_1$. Two of them are obtained from
the two factorizations,
\begin{equation}
  \label{eq:7}
  P = (t-h_1)(t-h_2) = (t-h'_1)(t-h'_2)
  \quad\text{with}\quad
  h_1,h_2,h'_1,h'_2 \in \HSet
\end{equation}
of $P$ over $\HSet$.\footnote{Motion polynomials over
  $\HSet$ always admit a finite number of factorizations.}  They are
\begin{align}
  \label{eq:A1}
  \tag{$A$}
  C &= (\xi - \eps\eta\qi)(t-h_1)(t-h_2) \\
  \label{eq:A2}
  \tag{$A'$}
    &= (\xi - \eps\eta\qi)(t-h'_1)(t-h'_2).
\end{align}
The open chains to each factorization consist of two revolute joints,
intersecting in the origin $p$, and one prismatic joint in direction
of $\qi$. The trajectory of $p$ is trivially a straight line.

Two further factorizations are of the shape
\begin{align}
  \label{eq:B1}
  \tag{$B$}
  C &= (t-r_1)(t-r_2)(t-s_1) \\
  \label{eq:B2}
  \tag{$B'$}
    &= (t-r'_1)(t-r'_2)(t-s_1)
\end{align}
with rotation quaternions $r_1$, $r_2$, $r'_1$, $r'_2 \in \DHSet$ and a
translation quaternion $s_1 \in \RSet + \eps\HSet$.

Finally, there are two factorizations with factors $t-r_1$, $t-r_2$ on
the left and factors $t-h_2$, $t-h'_2$ on the right:
\begin{align}
  \label{eq:C1}
  \tag{$C$}
  C &= (t-r_1)(t-s_2)(t-h_2) \\
  \label{eq:C2}
  \tag{$C'$}
    &= (t-r'_1)(t-s'_2)(t-h'_2).
\end{align}
Here, the translation quaternions are $s_2$ and $s'_2$. In each chain,
the last revolute axis (corresponding to the factor on the right)
contains the origin $p$ of the moving frame.

Assuming that the two factorizations in \eqref{eq:7} are really
different, a suitable combination of the factorizations
\eqref{eq:A1}--\eqref{eq:C2} results in spatial linkages with a
straight line trajectory. We will have a closer look at the manifold
relations between the involved joint axes. This will deepen our
geometric understanding of these linkage classes and provide us with
necessary conditions on the linkage's Denavit-Hartenberg parameters.

To begin with, it must be noted that not every combination of two open
chains resulting from the factorizations \eqref{eq:A1}--\eqref{eq:C2}
is admissible for the construction of overconstrained, single looped
linkages with one degree of freedom. In order to avoid ``dangling''
links, we must not combine two factorizations with the same factor at
the beginning or at the end. Hence, we have only four essentially
different admissible pairings:
\begin{equation*}
  \text{$A$--$B$},\
  \text{$A$--$C'$},\
  \text{$B$--$C'$},\
  \text{and}\
  \text{$C$--$C'$}.
\end{equation*}
Non-admissible pairings do not give suitable linkages but information
on joint axes. If two factorizations have a common factor at the
beginning or the end, the remaining factors can be assembled into a
closed linkage with four joints. Consider, for example, the factors
\eqref{eq:A1} and \eqref{eq:C1}. Their closure equation simplifies to
\begin{equation*}
  \begin{aligned}
    1 &\peq (\xi - \eps\eta\qi)(t-h_1)(t-h_2)(\cj{t-h_2})(\cj{t-s_2})(\cj{t-r_1}) \\
      &\peq (\xi - \eps\eta\qi)(t-h_1)(\cj{t-s_2})(\cj{t-r_1}) \\
      &= (\xi - \eps\eta\qi)(t-h_1)(t-\cj{s_2})(t-\cj{r_1}).
  \end{aligned}
\end{equation*}
Hence, the axes of the pair $(h_1,r_1)$, and also that of
$(h'_1,r'_1)$, $(h_2,r_2)$, and $(h'_2,r'_2)$, are parallel because
they are revolute axes in overconstrained RPRP linkages. A similar
argument shows that the axes to $r_1$, $r'_1$, $r'_2$, and $r_2$
define a Bennett linkage. Finally, the axes to $\eps\qi$, $h_1$,
$h'_1$, $h_2$, and $h'_2$ intersect in the point $p$ whose trajectory
is a straight line. These observations are responsible for special
geometric features of the admissible linkages.
\begin{description}
\item[Type $A$--$B$:] The linkage is of type PRRPRR. The second and third
  axes intersect. The second and sixth axis and the third and fifth
  axis are parallel.
\item[Type $A$--$C'$:] In this linkage, three consecutive revolute
  axes (corresponding to $h_1$, $h_2$, $h'_2$) intersect so that we
  may view it as PSPR linkage. However, because of \eqref{eq:7} we have
  \begin{equation*}
    (t-h_1)(t-h_2)(\cj{t-h'_2}) \peq t-h'_1
  \end{equation*}
  and the spherical joint can actually be replaced by a revolute
  joint. It has to be noted that this replacement cuts away the end
  effector and, thus, changes the end effector motion. One consequence
  of this mental collapsing of S joint into R joint are the angle
  equalities
  \begin{equation*}
    \sphericalangle(\qi, r_1) = \sphericalangle( r_1, s_2),\quad
    \sphericalangle(\qi,r'_1) = \sphericalangle(r'_1,s'_2).
  \end{equation*}
  which are known to hold for the corresponding RPRP linkages. Here,
  the angle between rotation and translation quaternions is to be
  understood as angle between their respective axis directions.
\item[Type $B$--$C'$:] This linkage of type RRPRPR contains a
  Bennett triple of revolute axes (axes one, two and six). 
\item[Type $C$--$C'$:] This is an RPRRPR linkages where the third and
  fourth axes intersect. An example is depicted in
  \autoref{fig:RPRRPR}. The linkage differs from Type $A$--$B$ in
  the linkage geometry and in the position of the link with straight
  line trajectory.
\end{description}

\begin{figure}
  \centering
  \includegraphics{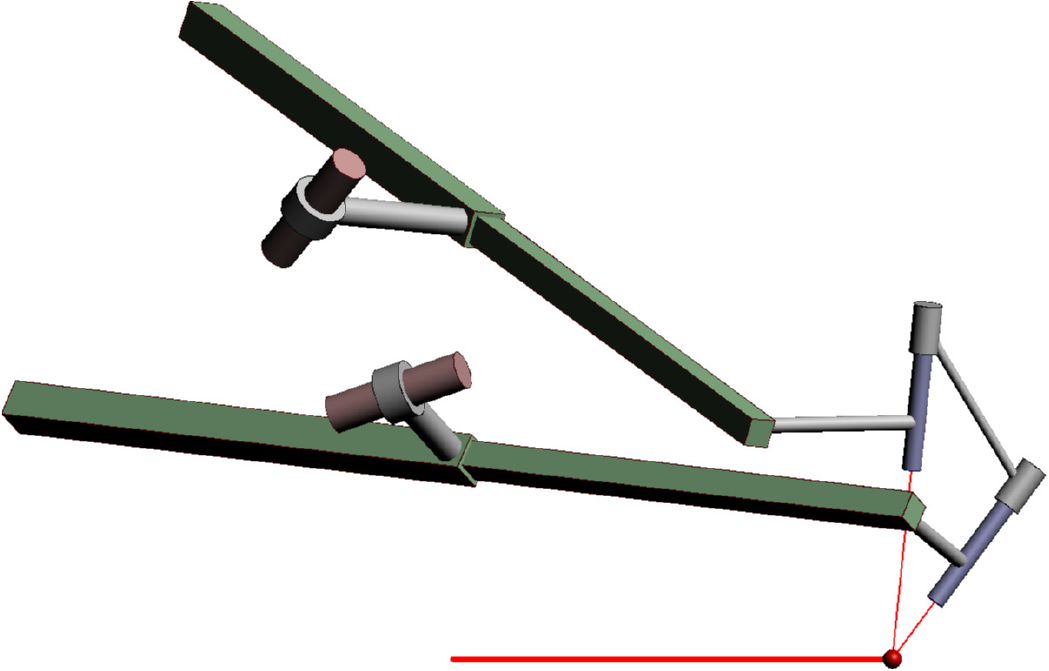}
  \caption{An RPRRPR linkage with a straight line trajectory}
  \label{fig:RPRRPR}
\end{figure}

\subsection{The cases of degree one}
\label{sec:cases-degree-one}

Now we turn to the case $\deg P = 1$ and start our discussion with the
sub-case $\deg \xi = 1$. The motion polynomial $C$ is of degree two
and it is well-known that its factorizations produce either Bennett
linkages or, in limiting cases, an RPRP linkage. The latter occurs
here because
\begin{equation*}
  C = (\xi - \eps\eta\qi)P
\end{equation*}
clearly is a factorization of $C$. The second factor, $P$, describes a
rotation about an axis through $p$, the first factor, $\xi - \eta\qi$,
describes a translation in direction of $\qi$. We omit the possible
computation of the second pair of revolute and prismatic joints as
this gives us no additional insight. Clearly, every point of either
rotation axis and in particular the point $p = (0,0,0)$ has a straight
line trajectory.

The remaining cases, $\deg P = 1$, $\deg \xi = 2$, and $\deg \eta = 0$
or $\deg \eta = 1$, can be discussed together. Motion polynomial and
norm polynomial are
\begin{equation*}
  C = (\xi - \eps\eta\qi)P
  \quad\text{and}\quad
  C\cj{C} = \xi^2 P\cj{P}.
\end{equation*}
We distinguish two sub-cases:

In the first case, the polynomial $\xi$ factors over the reals. Then,
by \autoref{th:1}, every closed linkage obtained from factorization of
$C$ has four prismatic and two revolute joints. The axes of the
revolute joints are necessarily parallel and the joint angles for
every parameter value $t$ sum to zero. For every fixed revolute joint
angle, the linkage admits a one-parametric translational motion along
a fixed line. Hence, it has two degrees of freedom and infinitely many
straight line trajectories.
  
In the second case, the polynomial $\xi$ is \emph{irreducible} over
the reals. Then every closed linkage obtained from factorization of
$C$ necessarily consists of only revolute joints which makes the
envisaged generation of a straight line trajectory even more
interesting.  It will turn out that this is only possible under very
special circumstances.

Setting $\xi = t^2 + x_1t + x_0$, $P = t - h$, and $\eta = y_1t + y_0$
with $h \in \HSet$ and $x_0$, $x_1$, $x_2$, $y_0$, $y_1 \in \RSet$, we
assume that $C$ factors as $C = C_1(t-k)$ with a rotation or translation
quaternion $k$. By \autoref{lem:1}, $k$ must be a zero of $C$. We set
$k = k_1 + \eps k_2$ with $k_1$, $k_2 \in \HSet$ and compute
\begin{equation*}
  0 = C(k) = P(k')\xi(k') +
    \eps\bigl( P(k_1)(k_1k_2 + k_2k_1 + x_1k_2) + k_2\xi(k_1) - \qi P(k_1)\eta(k_1) \bigr).
\end{equation*}
In order for the primal part to vanish, we have either $P(k_1) = 0$ or
$\xi(k_1) = 0$. In the former case, we have $k_1 = h$ and the dual
part vanishes only if $k_2 = 0$ or $\xi(k_1) = 0$. If $k_2=0$, we have
$C_1=\xi - \qi\eta\eps$ and, by \autoref{th:4} in the appendix, $C_1$
admits no further factorization. Hence, we can assume $\xi(k_1) = 0$
in any case. This implies $x_1 = - (k_1 + \cj{k_1})$ and
$x_0 = k_1\cj{k_1}$.

The quaternion zeros of a quadratic equation are completely described
by \cite[Thereom~2.3]{huang02}. Because $\xi$ is irreducible over
$\RSet$ and $\xi(k_1) = 0$, we have
\begin{equation}
  \label{eq:8}
  k_1 = \frac{1}{2}(-x_1 + w(s_1\qi + s_2\qj + s_3\qk))
\end{equation}
where $w = \sqrt{4x_0-x_1^2}$ and $(s_1,s_2,s_3) \in S^2$. Given $k_1$
as in \eqref{eq:8}, the dual part $k_2$ of $k$ has to satisfy
\begin{equation*}
  P(k_1)(k_1k_2 + k_2k_1 + x_1k_2) - \qi P(k_1)\eta(k_1) = 0
  \quad\text{and}\quad
  k_1\cj{k_2} + k_2\cj{k_1} = 0.
\end{equation*}
Because of $\cj{k_2} = -k_2$, the second equation
implies $k_1k_2 = k_2\cj{k_1}$. We plug this in the first equation and
find
\begin{equation*}
  0 = P(k_1)(k_2(\underbrace{\cj{k_1} + k_1}_{-x_1}) + x_1k_2) - \qi P(k_1)\eta(k_1)
    = -\qi P(k_1)\eta(k_1).
\end{equation*}
This is only possible if $P(k_1) = 0$. Hence, we have $k_1 = h$, $x_1
= -h - \cj{h}$ and $x_0 = h\cj{h}$ or, equivalently, $P\cj{P} = \xi$.
We will prove in \autoref{th:2} below that the motion parameterized by
$C$ is the well-known Darboux motion, see \cite{krames37,lee12} or
\cite[Chapter~9, \S3]{bottema90}. This is the unique non-planar,
non-spherical and non-translational motion with only planar
trajectories. It is the composition of a planar elliptic motion and a
harmonic oscillation perpendicular to the plane of the elliptic
motion. Its trajectories are ellipses with the same major axis length
and some trajectories indeed degenerate to straight line segments.

\begin{theorem}
  \label{th:2}
  Unless $h$ lies in the linear span of $\qj$ and $\qk$, the motion
  parameterized by $C = \xi P - \qi\eta\eps P \in \DHSet[t]$ with $P =
  t - h \in \HSet[t] \setminus \RSet[t]$, $\xi = P\cj{P}$, $\eta \in
  \RSet[t]$, $\eta \neq 0$, $\deg \eta \le 1$ is a \emph{Darboux
    motion.}
\end{theorem}

\begin{proof}
  Using $P\cj{P} = \xi$, we compute the parametric equation
  \begin{equation*}
    \frac{2\eta\qi}{\xi} + \frac{P(x\qi + y\qj + z\qk)\cj{P}}{\xi}
  \end{equation*}
  for the trajectory of a point $(x, y, z)$. We see that all
  coordinate functions are at most quadratic. Hence, all trajectories
  are planar. Since $\eta$ is different from zero, it is not a
  spherical motion. Because of our assumptions on $h$, it is no planar
  or translational motion.
\end{proof}

We already excluded translational end-effector motions from our
considerations and can therefore focus on the factorization and
linkage construction for Darboux motions, given by $C$ as in
\autoref{th:2}. Algorithmic factorization, as explained in
\autoref{sec:motion-polynomials} fails for Darboux motions. Thus, a
special discussion is necessary. We already saw previously, that right
factors are necessarily of the shape $t - (h + \eps k_2)$.
Conversely, any linear polynomial of that shape is really a right
factor. The factorization is $C = C_1(t - (h + \eps k_2))$ where
\begin{equation}
  \label{eq:9}
  C_1 = \xi + \eps D
\end{equation}
and, with $k_2 = a\qi + b\qj + c\qk$,
\begin{multline}
  \label{eq:10}
  D = ((a - y_1)\qi + b\qj + c\qk)t -  a h_1 + b h_2 + h_3 c \\
                                         -  (h_0 a + h_2 c - h_3 b + y_0) \qi
                                         -  (h_0 b - h_1 c + h_3 a) \qj
                                         -  (h_0 c + h_1 b - h_2 a) \qk.
\end{multline}

The factorizability of $C_1$ is discussed in \autoref{th:4} in the
appendix. Summarizing the results there, we can say the following:
\begin{itemize}
\item The motion parameterized by $C_1$ is a planar translational
  motion whose trajectories are rational of degree two (or less).
\item It admits factorizations if and only if it parameterizes a
  circular translation. In this case, it admits infinitely many
  factorizations, corresponding to the multiple generation of a
  circular translation by parallelogram linkages.
\item A criterion for circular translations is $\xi \peq D\cj{D}$.
\end{itemize}

Thus, we only have to answer, under which conditions on $a$, $b$, $c$
Equation~\eqref{eq:9} is a circular translation or, equivalently, $\xi$ is a
factor of $D\cj{D}$. The latter gives convenient linear equations for
$a$, $b$, $c$. Writing
\begin{equation*}
  D\overline{D} = z_2t^2 + z_1t + z_0
\end{equation*}
where $D$ is as in \eqref{eq:10}, the linear system to solve is
\begin{equation}
  \label{eq:11}
  z_0x_1 - z_1x_0 = z_0x_2 - z_2x_0 = z_1x_2 - z_2x_1 = h_1a + h_2b + h_3c = 0.
\end{equation}
This overconstrained system has a matrix $M$. The greatest common
divisor of all $3 \times 3$ minors of $M$ is
\begin{equation*}
  \Delta \coloneqq 4(h_2^2+h_3^2)((h_0y_1+y_0)^2+y_1^2(h_1^2+h_2^2+h_3^2)).
\end{equation*}
Again, we need to distinguish two cases:

If $h_2 = h_3 = 0$, the motion is the composition of a rotation about
$\qi$ and a translation in direction $\qi$, that is, a \emph{vertical
  Darboux motion.} Because $P$ is not a real polynomial, $h_1$ is
different from zero and we necessarily have $a = 0$. This leaves us
with three conditions on the solubility:
\begin{equation*}
  y_1(h_0y_1 + y_0) = y_0((h_0^2+h_1^2)y_1+h_0y_0) = (h_0^2+h_1^2)y_1 + y_0^2 = 0.
\end{equation*}
A straightforward discussion shows that either $h_1$ or $y_1$ vanish.
But both, $h_1 = 0$ and $y_1 = 0$ have been excluded previously.
Hence, the vertical Darboux motion allows no factorizations into the
product of three linear factors.

If $h_2$ and $h_2$ are not both zero, $\Delta$ cannot vanish and the
system \eqref{eq:11} has the unique solution
\begin{equation*}
  a = \frac{y_1}{2},\quad
  b = \frac{y_0h_3 + y_1(h_0h_3-h_1h_2)}{2(h_2^2+h_3^2)},\quad
  c = \frac{y_0h_2 + y_1(h_0h_2+h_1h_3)}{2(h_2^2+h_3^2)}.
\end{equation*}
In other words, there is precisely one admissible choice for $k_2$
such that \eqref{eq:9} is a circular translation and admits
infinitely many factorizations. Thus, we have proved

\begin{theorem}
  \label{th:3}
  A non-vertical Darboux motion, parameterized by $C$ as in
  \autoref{th:2}, admits infinitely many factorization into linear
  motion polynomials. The first two factors on the left describe the
  same circular translation, the right factor is the same for all
  factorizations.
\end{theorem}

Closed loop linkages for the generation of vertical Darboux motions
are described in \cite{lee12}. Here, it seems that we closely missed
the possibility to construct a closed loop linkage with one degree of
freedom and only revolute joints that generates a general
(non-vertical) Darboux motion. Though we managed to factor the
non-vertical Darboux motion in infinitely many ways, we may not form a
linkage with one degree of freedom from two factorizations as they
have the right factor in common. Nonetheless, there is a way out of
this. It requires a ``multiplication trick'' which will be
investigated in more detail and generality in a forthcoming
publication. Here, we confine ourselves to present the basic idea at
hand of a concrete example.

We consider the Darboux motion $C = \xi P - \qi \eta \eps P \in
\DHSet[t]$ with
\begin{equation*}
  \xi = t^2 + 1,\quad
  \eta = \frac{5}{2} t - \frac{3}{4},\quad
  P = t-h\quad\text{and}\quad
  h=\frac{7}{9} \qi - \frac{4}{9} \qj + \frac{4}{9} \qk.
\end{equation*}
As seen above, this give us a first factorization $C = Q_1 Q_2 Q_3$, where
\begin{equation*}
  \begin{aligned}
    Q_1 & = t - \frac{7}{9} \qi - \frac{4}{9} \qj + \frac{4}{9} \qk - \frac{5}{4} \eps \qi + \frac{43}{64} \eps \qj - \frac{97}{64} \eps \qk, \\
    Q_2 & = t + \frac{7}{9} \qi + \frac{4}{9} \qj - \frac{4}{9} \qk,                                                                          \\
    Q_3 & = t - \frac{7}{9} \qi + \frac{4}{9} \qj - \frac{4}{9} \qk - \frac{5}{4} \eps \qi - \frac{43}{64} \eps \qj + \frac{97}{64} \eps \qk.
  \end{aligned}
\end{equation*}
In order to obtain a second factorization, we first set the right
factor to $Q_4 \coloneqq P$ and compute $C_1$ such that $C = C_1Q_4$:
\begin{equation*}
  C_1 = t^2 + 1 - \eps \qi \left( \frac{5}{2} t - \frac{3}{4} \right).
\end{equation*}
The motion polynomial $C_1$ parameterizes a translation in constant
direction. According to \autoref{th:4} in the appendix, it cannot be
written as the product of two linear motion polynomials. However,
after multiplying $C_1$ by $t^2 + 1$, it actually has infinitely many
factorizations into products of \emph{three} motion polynomials, one of
them being $C' (t^2 + 1)=Q_7 Q_6^2 Q_5,$ where
\begin{equation*}
  \begin{aligned}
    Q_7 & = t - \qj - \frac{5}{4} \eps \qi - \frac{3}{8} \eps \qk, \\
    Q_6 & = t + \qj,                                               \\
    Q_5 & = t - \qj - \frac{5}{4} \eps \qi + \frac{3}{8} \eps \qk.
  \end{aligned}
\end{equation*}
The multiplicity of the middle factor $Q_6$ is no coincidence but
inherent in the structure of the factorization problem at hand. The
kinematic structure to this factorization is an open 4R chain with
coinciding second and third axis, that is, actually just a 3R
chain. Because $C = Q_1Q_2Q_3$ and $\xi C = Q_7Q_6^2Q_5Q_4$ are
projectively equal, we can combine these two factorizations to form a
7R linkage where each rotation is defined by $Q_i$, $i =
1,\ldots,7$. It can be seen that the axes of $Q_1$, $Q_2$ are
parallel, as are the axes of $Q_3,\ Q_4$ and $Q_5,\ Q_6,\
Q_7$. Moreover, all joint angles are the same -- a property that has
not yet been observed in non-trivial linkages obtained from motion
polynomial factorization.

To complete above construction, we should check that the configuration
space of the 7R linkage is really a curve. A Gröbner basis computation
reveals that this is indeed the case. Note that the configuration
curve contains several components, also components of higher genus.
One component corresponds to the rational curve parameterized by $C$.
Thus, we have indeed constructed a 7R linkage whose coupler motion is
a non-vertical Darboux motion. In \autoref{fig:darboux}, we present
three configurations of this linkage in an orthographic projection
parallel to~$\qj$. We can observe the parallelity of axes and
constancy of one direction during the coupler motion.

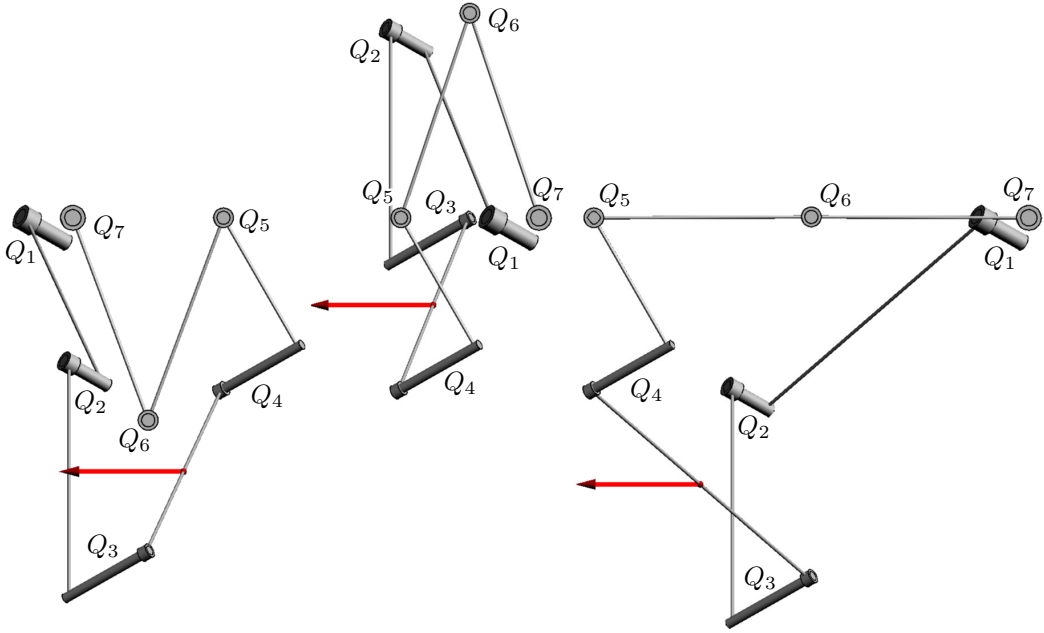
\begin{figure}
  \centering
  \begin{overpic}[scale=0.85,trim=170 0 0 0,clip]{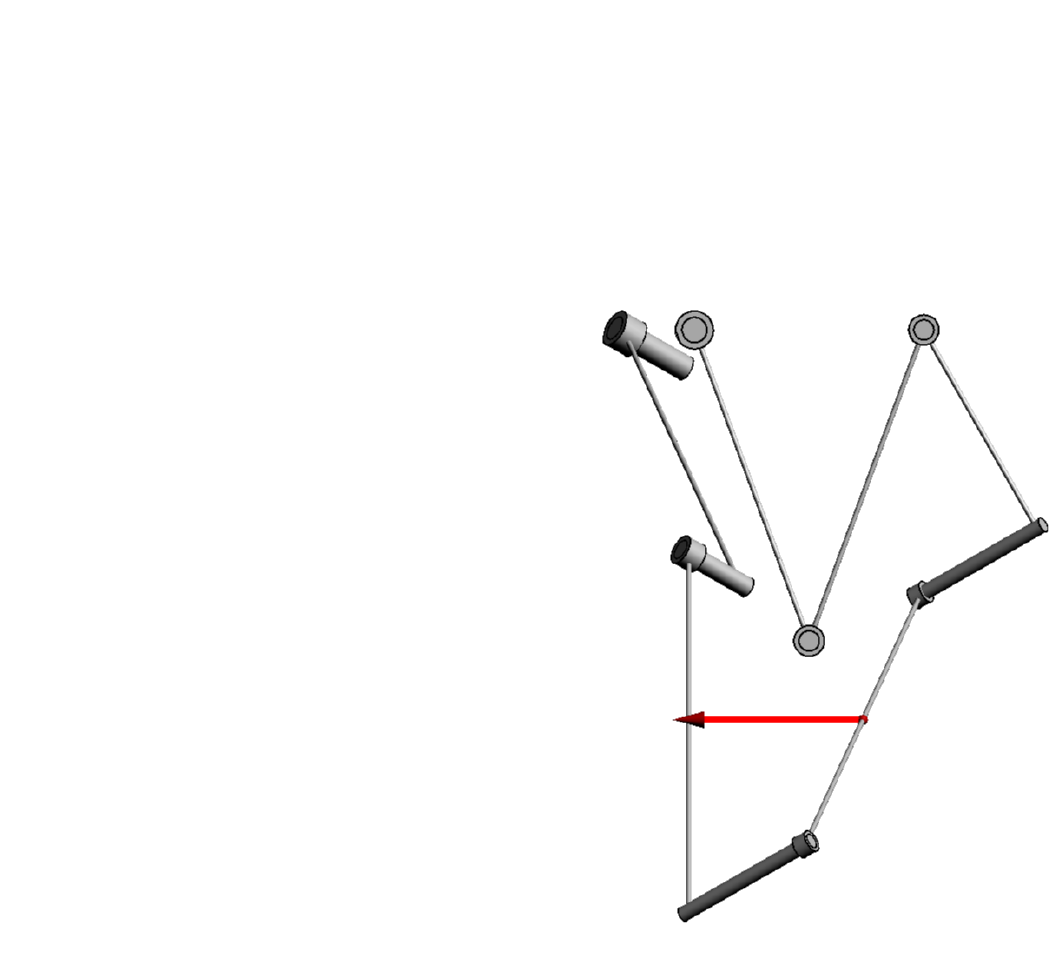}
    \small
    \put(0,59){$Q_1$}
    \put(11,35){$Q_2$}
    \put(13,12){$Q_3$}
    \put(39,36){$Q_4$}
    \put(37,64){$Q_5$}
    \put(18,29){$Q_6$}
    \put(14,63){$Q_7$}
  \end{overpic}
  \begin{overpic}[scale=0.85,trim=100 0 90 0,clip]{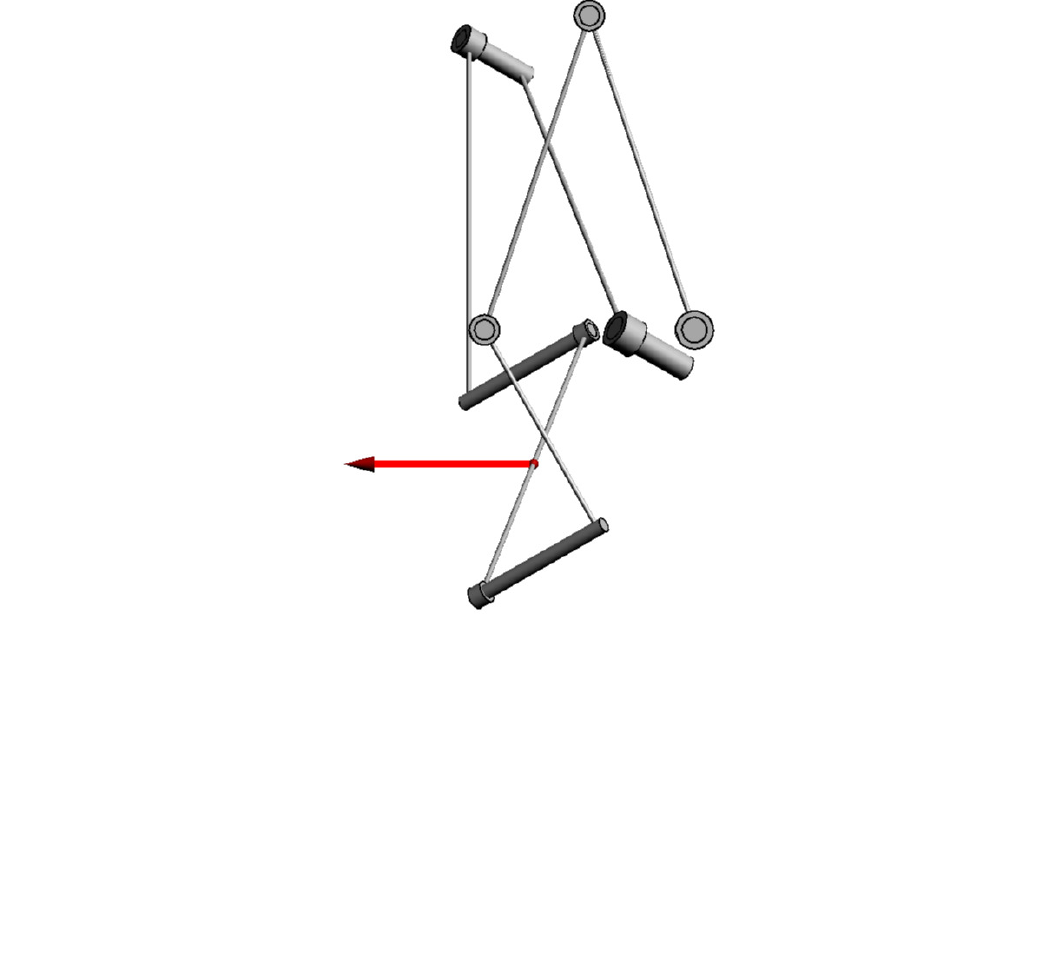}
    \small
    \put(28,58){$Q_1$}
    \put(6,91){$Q_2$}
    \put(18,67){$Q_3$}
    \put(21,38){$Q_4$}
    \put(8,68.5){\contour{white}{$Q_5$}}
    \put(28,96){$Q_6$}
    \put(35,69){$Q_7$}
  \end{overpic}
  \begin{overpic}[scale=0.85,trim=0 0 90 0,clip]{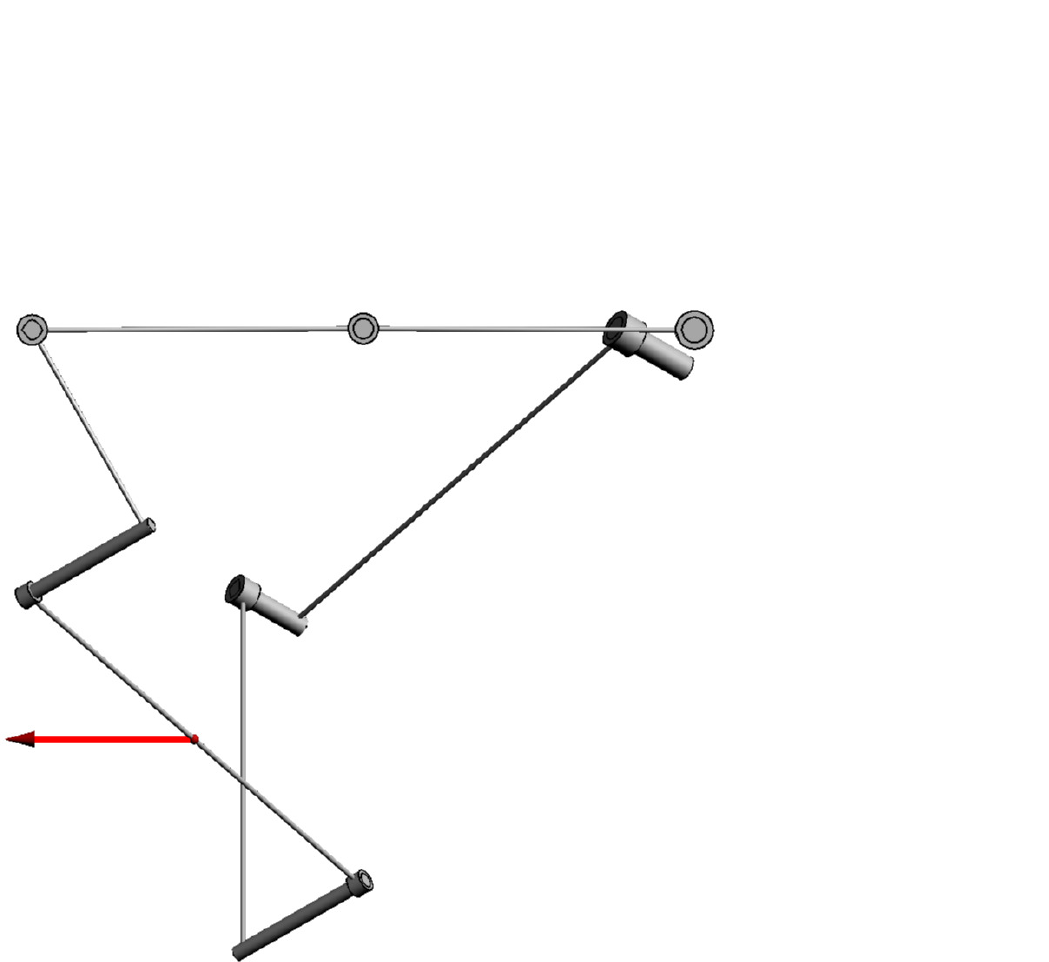}
    \small
    \put(65,58){$Q_1$}
    \put(26,31){$Q_2$}
    \put(27,7){$Q_3$}
    \put(9,37){$Q_4$}
    \put(4,68){$Q_5$}
    \put(39,68){$Q_6$}
    \put(68,69){$Q_7$}
  \end{overpic}
  \caption{A 7R linkage that generates a non-vertical Darboux motion.}
  \label{fig:darboux}
\end{figure}

\section{Conclusions and future research}
\label{sec:conclusions}

We have studied spatial straight line linkages obtained by factorizing
a cubic motion polynomial. The mobility and straight line property of
some of the resulting linkages can be explained geometrically while
for others the explanation remains algebraic. In the course of this
investigation, we showed that a Darboux motion can be decomposed into
a circular translation and a rotation and we presented one particular
example of a 7R Darboux linkage. A closer investigation of the used
``multiplication trick'' is left to a forthcoming publication.

Another natural step is to study general trajectory generation in
relation to the factorization of motion polynomials. We are already in
a position to announce concrete and promising results in this
direction.

As already mentioned in the introduction, the engineering relevance of
these linkages is probably limited. The present investigation should
be rather seen as an exercise in factorization of motion polynomials
and a demonstration of what it is capable of. We expect more
interesting and applicable linkages to arise from the factorization of
motion polynomials in other constraint varieties. Already a cursory
glance at the descriptions of constraint varieties in \cite{selig11}
shows that there is plenty of room for further investigations.

\appendix

\section{Factorization of quadratic translational motions}
\label{sec:translational}

In this appendix we provide a complete discussion of the
factorizability of translational motions that are parameterized by a
quadratic motion polynomial into the product of two linear motion
polynomials. We prove two theorems the first of which is often
referenced in the preceding text. The second theorem is not used in
this paper. We present it for the sake of completeness and because it
may be interesting in its own right.

Throughout this section, $C = \xi + \eps D$ is a monic, quadratic
motion polynomial with $\xi \in \RSet[t]$, $\deg \xi = 2$ and $D \in
\HSet[t]$. It is our aim to give a complete description of all
possibilities to write $C$ as $C = (t - h)(t - k)$ with rotation or
translation quaternions $h$, $k \in \DHSet$.

Lets start with some basic properties of the motion $C$. Because
$\xi$, the primal part of $C$, is a real polynomial, the motion is
translational. Because $C$ is of degree two and monic, the degree of
$D$ is at most one. Moreover, $C\cj{C} = \xi(\xi + \eps(D + \cj{D}))
\in \RSet[t]$ implies $\cj{D} = -D$. Conversely, any translational
motion of degree two can be written in that way.

The trajectory of the coordinate origin can be parameterized as
$x_0^{-1}(x_1,x_2,x_3)$ with polynomials $x_i \in \RSet[t]$, given by
\begin{equation}
  \label{eq:12}
  x_0 + \eps(x_1\qi + x_2\qj + x_3\qj) =
  \xi(\xi - 2\eps(\cj{D} - D)) =
  \xi(\xi - 2\eps D) \peq
  \xi - 2\eps D.
\end{equation}
We see that this trajectory is rational of degree two at most. Hence,
the motion under investigation is a \emph{planar, curvilinear
  translation.}

\begin{theorem}
  \label{th:4}
  Let $C = \xi + \eps D$ be a monic, quadratic motion polynomial with
  \emph{irreducible} $\xi \in \RSet[t]$, $\deg \xi = 2$, $D \in \HSet[t]$.
  Then the following statements are equivalent:
  \begin{enumerate}
  \item There exist two rotation quaternions $h$, $k \in \DHSet$ such
    that $C = (t-h)(t-k)$.
  \item There exist infinitely many rotation quaternions $h$, $k \in
    \DHSet$ such that $C = (t-h)(t-k)$.
  \item The motion polynomial $C$ parameterizes a circular
    translation.
  \item The polynomial $\xi$ divides $D\cj{D}$. (This implies $\xi
    \peq D\cj{D}$.)
  \end{enumerate}
\end{theorem}

\begin{proof}
  1 $\implies$ 4: Write $h = h_1 + \eps h_2$, $k = k_1 + \eps k_2$
  with rotation quaternions $h_1$, $h_2$, $k_1$, $k_2
  \in \HSet$. By equating the primal part of $(t-h)(t-k)$ with $\xi$
  we find $h_1 + k_1 \in \RSet$ and $h_1k_1 \in \RSet$. This is only
  possible if $k_1 = \cj{h}_1$. Let us write, for simplicity, $p
  \coloneqq h_1 = \cj{k_1}$. Then $\xi = t^2 - (p + \cj{p})t + p\cj{p}
  = (t - p)(t - \cj{p})$.

  Because $k = \cj{p} + \eps k_2$ is a rotation quaternion, we have
  $pk_2 = -\cj{k_2}\cj{p} = k_2\cj{p}$ (Study condition) and hence
  \begin{equation*}
    (t-p)k_2 = k_2t - pk_2 = k_2t - k_2\cj{p} = k_2(t-\cj{p}).
  \end{equation*}
  Using this, the dual part of
  $(t - h_1 - \eps h_2)(t - k_1 - \eps k_2)$ can be written as
  \begin{equation*}
    D =
    -(h_2(t - \cj{p}) + (t - p)k_2) =
    -(h_2 + k_2)(t - \cj{p}).
  \end{equation*}
  Compute now
  \begin{equation*}
    D\cj{D} = (h_2 + k_2)(t - \cj{p})(t - p)(\cj{h_2} + \cj{k_2})
            = \xi q\cj{q}
  \end{equation*}
  with $q = h_2 + k_2$. Thus, $\xi$ is, indeed, a factor
  of~$D\cj{D}$.

  4 $\implies$ 3: We already know that $C$ describes a curvilinear
  translation with rational quadratic trajectories given by
  \eqref{eq:12}. The trajectory of the coordinate origin (and hence
  all other trajectories) are circles if its points at infinity lie
  on the absolute conic of Euclidean geometry. Algebraically this
  means that $x_0 = \xi$ divides $x_1^2 + x_2^2 + x_3^2 = 4D\cj{D}$.
  But this is precisely the assumption.

  3 $\implies$ 2: A circular translation occurs in infinitely many
  ways as coupler motion of a parallelogram linkage. This linkage is
  composed of two 2R chains, each corresponding to one of infinitely
  many factorizations of $C$.\footnote{This can also be verified at
    hand of a concrete example. The circular translation $C = 1+t^2 -
    \eps (\qi + \qj t)$ allows the factorizations $C = (t - \qk - \eps
    (-a\qi + (1-b)\qj)) (t + \qk - \eps ( a\qi + b \qj)))$ with $a, b
    \in \RSet$.}

  The trivial final implication (2 $\implies$ 1) completes the proof.
\end{proof}

\begin{remark}
  By \autoref{th:1}, translation quaternions cannot occur in the
  factorization of $C$ if $\xi$ is irreducible. Hence \autoref{th:4}
  gives all factorizations in the case of irreducible $\xi$.
\end{remark}

\begin{theorem}
  \label{th:5}
  Let $C = \xi + \eps D$ be a monic, quadratic motion polynomial with
  \emph{reducible} $\xi \in \RSet[t]$, $\deg \xi = 2$, $D \in \HSet[t]$.
  \begin{itemize}
  \item If $\xi$ has no root of multiplicity two, there exist two
    translation quaternions $h$, $k$ such that any factorization of
    $C$ into the product of two translation quaternions is either $C =
    (t-h)(t-k)$ or $C = (t-k)(t-h)$.
  \item If $\xi$ has a root $\lambda$ of multiplicity two, $C$ can be
    written as the product of linear translation polynomials if and
    only if it is of the shape $C = (t-\lambda)^2 +
    \eps(t+\lambda)d_1$ with $d_1 \in \HSet$, $d_1 \neq 0$, $\cj{d_1}
    = -d_1$. In this case, it is a translation in constant direction
    and infinitely many factorizations exist.
  \end{itemize}
\end{theorem}

\begin{proof}
  Write $h = h_0 + \eps h_2$, $k = k_0 + \eps k_2$ with $h_0$, $k_0
  \in \RSet \setminus \{0\}$, $h_2$, $k_2 \in \HSet$ with $\cj{h_2} =
  -h_2$, $\cj{k_2} = -k_2$ and compare coefficients of $C$ and $(t -
  h)(t - k)$. The reals $h_0$ and $k_0$ are determined as (real) roots
  of $\xi$. Provided $\xi$ has no double root, $h_2$ and $k_2$ are
  uniquely determined as
  \begin{equation*}
    h_2 = \frac{1}{k_0-h_0}(d_0-h_0d_1),\quad
    k_2 = \frac{1}{k_0-h_0}(k_0d_1 - d_0),
  \end{equation*}
  where $D = d_1t + d_0$. The factors $t-h$ and $t-k$ commute and the
  first claim follows.

  Assume now $h_0 = k_0 \eqqcolon \lambda$. Then, we have to solve
  \begin{equation*}
    h_2 + k_2 = d_1,\quad
    \lambda(h_2 + k_2) = -d_0.
  \end{equation*}
  A necessary and sufficient condition for existence of a solution is
  $\lambda d_1 = -d_0$, that is $D = (t + \lambda)d_1$.  The motion
  is, indeed, a translation in constant direction. Given $C =
  (t-\lambda)^2 + \eps(t + \lambda)d_1$, we can set $h_0 = k_0 =
  \lambda$ and determine $h_2$ and $k_2$ in infinitely many ways such
  that $h_2 + d_2 = d_1$ and $C = (t-h)(t-k) = (t-k)(t-h)$ are
  factorizations of~$C$.
\end{proof}

\begin{remark}
  Reducability of $\xi$ implies that in every factorization of $C$ a
  translation polynomial occurs. But then the second factor must also
  be a translation polynomial. Hence, \autoref{th:5} gives all
  factorizations for the case of reducible~$\xi$.
\end{remark}

\autoref{th:4} and \ref{th:5} describe all cases that admit
factorizations with two linear factors. But there also exist motion
polynomials $C = \xi + \eps D$ that do not admit such a
factorization. One example is
\begin{equation*}
  C = t^2 + 1 + \eps \qi t.
\end{equation*}
The primal part of $C$ is irreducible over $\RSet$ but the motion is a
translation in fixed direction. By \autoref{th:4}, $C$ cannot be
written as the product of two linear motion polynomials. By a variant
of the ``multiplication trick'' that we already used in
\autoref{sec:cases-degree-one} we can still find infinitely many
factorizations of $C(t^2+1)$ with four linear motion
polynomials. Enforcing identical consecutive factors, as at the end of
\autoref{sec:cases-degree-one}, will then produce a Sarrus
linkage. But this is again another story and shall be left to the
forthcoming publication.

\section*{Acknowledgments}

This work was supported by the Austrian Science Fund (FWF):
P~23831-N13, P~26607, and W1214-N15, project DK9.

\bibliographystyle{plainnat}
\bibliography{sl6R}

\end{document}